\documentclass{article}

\usepackage{graphicx}
\usepackage{tikz}
\usepackage{tikz-cd}
\usepackage{amsmath}
\usepackage{amsfonts}
\usepackage{amsthm}
\usepackage{amssymb}
\usepackage{tabularx}
\usepackage{booktabs}
\usepackage{url}
\usepackage[colorlinks=true, linkcolor=blue, citecolor=red]{hyperref}

\newcommand{\Z}{\mathbb Z}

\DeclareMathOperator{\depth}{depth}
\DeclareMathOperator{\pd}{pd}
\DeclareMathOperator{\Ass}{Ass}
\DeclareMathOperator{\im}{im}
\DeclareMathOperator{\Ext}{Ext}

\newtheorem{theorem}{Theorem}[section]
\newtheorem{lemma}[theorem]{Lemma}
\newtheorem{proposition}[theorem]{Proposition}
\newtheorem{corollary}[theorem]{Corollary}

\theoremstyle{definition}
\newtheorem{definition}[theorem]{Definition}

\title{The local homological structure of generalized splines}
\author{Kyle Stoltz}
\date{}

\begin{document}

\maketitle

\begin{abstract}
Generalized splines are a simultaneous generalization of GKM theory---which studies equivariant cohomology---and classical splines, which provide piecewise approximations of functions. Generalized splines can also be understood via schemes, with the interpolation constraints---or so-called GKM-condition---encoded by gluing along certain closed subschemes. This view provides a local--global principle, with the local pictures retaining the generalized spline structure. Consequently, the behavior of generalized splines over local rings controls certain global phenomena, such as projectivity and often freeness.

We introduce an interface between the homological study of local rings and the combinatorial study of generalized splines. We identify precisely how the generalized spline structure coordinates with the existing homological local ring machinery. This is accomplished by two exact sequences that provide a regulatory structure on the local cohomology of a generalized spline module. As an application, we use this to prove that for any edge-labeled graph $G$ with principal ideal labels, and any Cohen--Macaulay ring $R$ of Krull dimension 2 at each maximal ideal, the module of splines $R_G$ is free, provided it has finite projective dimension. As a special case, this implies every generalized spline module over $k[x,y]$ with principal edge labels is free.
\end{abstract}

\section{Introduction}
Generalized splines sit at the interface of GKM-theory, which studies equivariant cohomology, and classical splines, objects that represent piecewise approximations of functions with various smoothness conditions. In a recent paper by the author \cite{stoltz2025scheme}, generalized splines were interpreted via scheme theory, which imported the use of local--global tools into the study of generalized splines. In particular, the generalized spline structure survives flat base change, which enables the use of the combinatorial properties of generalized splines after restriction to a local picture.

There are now two parallel research directions for the local--global approach. The first is to study how the combinatorial properties simplify (or do not simplify) upon flat base change. This is useful to consider because if the edge-labeled graph degenerates locally in a suitable way, that can be leveraged to extract global results. The second approach is to study how generalized splines behave after flat base change, and in particular, over local rings. Thus, while degeneration describes how the edge-labeled graph changes after flat base change, the local ring theory itself controls what degenerations are possible. In this paper, we focus on the behavior of generalized splines over local rings.

There are two natural ways to study generalized splines over local rings. The first is by purely combinatorial means, and the second is homological. In this paper, we focus on the homological view. In particular, the aim is to explain the articulation between the known homological properties of local rings, and the combinatorial properties of generalized splines. That is, we study precisely what the generalized spline hypothesis affords us homologically over local rings. 

Historically, this mirrors the early work in classical splines initiated by Billera and Rose, who first used a local--global argument for freeness based on stars \cite[Theorem 2.3]{billera1991modules}. Later, Schenck and Stillman advanced this approach with the use of local cohomology and depth estimates to analyze spline modules \cite{schenck1997local}. In the field of generalized splines, Alt{\i}nok and Sar{\i}o\u{g}lan \cite{altinok2024multivariate} developed a homological approach to multivariate generalized splines using syzygies, while Michael DiPasquale set up a homological framework for generalized splines using an exact sequence \cite{dipasquale2017generalized}. The approach taken here follows the homological direction already initiated for generalized splines, but with a focus on local rings. This is motivated by the interpretation of generalized splines as schemes, which facilitates a local--global principle intrinsic to the ring itself \cite[Main Theorem]{stoltz2025scheme}. 

To briefly summarize the main results: there are two exact sequences that control the homological behavior of generalized splines over local rings. The first is the potential exact sequence, for which the splines are the kernel and the cosplines are the cokernel, which was first studied by DiPasquale  \cite[Lemma 2.2]{dipasquale2017generalized}. The second is the copotential exact sequence, which has the cosplines as the kernel, and the failed cosplines as the cokernel. The objects of these two exact sequences are closely related to the vertex and edge structure of the associated edge-labeled graph, and thereby provide a convenient control point for proofs. To exemplify how to appropriately use the exact sequences, we prove a general theorem on Cohen--Macaulay rings with certain properties. As a special case, the theorem shows that any multivariate generalized spline module in two variables with principal edge labels is free.

\section{The Basic Homology of Local Rings}

First, we collect several results in the literature that naturally attach to the generalized spline structure. This section does not represent new results, but rather provides a repository for the major theoretical tools most clearly suited for working with generalized spline modules. We generally follow the standard notation found in the Stacks Project \cite{stacks-project}. Throughout this section, all $R$-modules are assumed to be finitely generated.

For a commutative Noetherian local ring $(R,\mathfrak m)$ and a nonzero finitely generated $R$-module $M$ of finite projective dimension, we have the Auslander--Buchsbaum formula \cite{brunsherzog}:
\[
  \pd_R(M) + \depth_R(M) = \depth_R(R).
\]

Here $\pd_R$ refers to the projective dimension of $M$, that is, the length of its projective resolution. When $\pd_R(M) = 0$, we have that $M$ is projective over $R$. The depth of an $R$-module $\depth(M)$ can be defined by the low degree cohomology of $\Ext$:
\[
\depth(M) = \min \{ i \in \Z \mid \Ext_R^i(R/\mathfrak m, M) \neq 0 \}
\]

where $\Ext$ is the derived functor cohomology of Hom \cite[Lemma 10.72.5]{stacks-project}. We can also work equivalently by local cohomology:
\[
\depth(M) = \min \{ i \in \Z \mid H_{\mathfrak m}^i(M) \neq 0 \}
\]

where $H$ is the derived functor cohomology of the $\mathfrak m$-torsion functor \cite[Lemma 47.11.1]{stacks-project}. Furthermore, for finite direct sums we have
\[
\depth(\bigoplus_{i \in I} M_i) = \min\{ \depth(M_i) \}_{i \in I}
\]

provided the $M_i$ are finitely generated modules \cite[Example 1.8]{mathew_homological}. For Cohen--Macaulay rings $R$, the Krull dimension of $R$ and the depth of $R$ as an $R$-module over itself coincide, which allows the $\depth(R)$ portion of the Auslander--Buchsbaum formula to be replaced with $\dim(R)$. Critically, for local Noetherian Cohen--Macaulay rings, the depth drop formula has $\depth R/\langle f \rangle = \depth R - 1$ for $\mathfrak m$-regular $f \in R$, where the depth of the quotient ring is as an $R$-module \cite[Proposition 2.9]{mathew_homological}. As we will see later, this minor and predictable depth drop is precisely why principal ideal labels are favorable to the presence of free generalized spline modules.

We also have the following mechanisms for depth control. Given a short exact sequence:
\[
0 \rightarrow N' \rightarrow N \rightarrow N'' \rightarrow 0
\]

we have \cite[Lemma 10.72.6]{stacks-project}

\[
\depth N' \geq \min\{ \depth(N), \depth(N'') + 1 \}
\]
\[
\depth N \geq \min\{ \depth(N'), \depth(N'') \}
\]
and
\[
\depth N'' \geq \min\{ \depth(N), \depth(N') - 1 \}.
\]

We also know that an $R$-module $M$ (again, $R$ is local) has depth $0$ if and only if $\mathfrak m$ is an associated prime of $M$, that is, every element of $\mathfrak m$ is a zero divisor on $M$ \cite[Example 1.2]{mathew_homological}. Furthermore, we have that the short exact sequence:
\[
0 \rightarrow N' \rightarrow N \rightarrow N'' \rightarrow 0
\]
implies
\[
\Ass(N') \subseteq \Ass(N),
\]

\[
\Ass(N) \subseteq \Ass(N') \cup \Ass(N''),
\]
and
\[
\Ass(N' \oplus N'') = \Ass(N') \cup \Ass(N'')
\]

per \cite[Lemma 10.63.3]{stacks-project}. We thus have two mechanisms of depth control for $N'$ via $N$ and $N''$. The first is to provide a lower bound for the depth of $N'$. The second is to observe that if $N$ has depth $1$, then $\mathfrak m$ is not an associated prime of $N$, and since $N'$ is a submodule, $\mathfrak m$ cannot be an associated prime of $N'$. This means $N'$ cannot have depth $0$. This allows $N$ to regulate the depth of $N'$ directly, albeit in a fairly restricted way. The other mechanism requires both $N$ and $N''$ working together, but provides a more robust floor on the depth of $N'$.

The local homological theory of generalized splines then proceeds as follows: we force the module of generalized splines (and closely related modules) into exact sequences, and then use those exact sequences to force $\depth(M) = \depth(R)$. When $M$ has finite projective dimension, the Auslander--Buchsbaum formula in turn implies $\pd_R(M) = 0$, and the local--global machinery carries this data to the global picture.

\section{The Homological Theory of Generalized Splines}
The previous section collected multiple results that show that the depth of a module can be studied via the exact sequences it appears in. This section is devoted to producing exact sequences that control the depth of generalized spline modules. The applications section will then use these exact sequences to extract the main theorem.

\begin{definition}
Let $R$ be a commutative ring, $(G, \alpha)$ a finite edge-labeled graph with vertex set $V$, edge set $E$, and an arbitrary fixed orientation. Denote the edge labels as $I_e := \alpha(e)$, $e \in E$. We define the following rings:

\begin{itemize}
\item $R^V := \bigoplus_{v \in V} R$, the \textbf{ring of potential splines}. We refer to elements of $R^V$ as \textbf{potential splines}.
\item $R^E := \bigoplus_{e \in E} R/I_e$, the \textbf{ring of potential cosplines}. We refer to elements of $R^E$ as \textbf{potential cosplines}.
\end{itemize}

The \textbf{potential morphism} $\rho_G: R^V \rightarrow R^E$ is induced by mapping $(r_v)_{v \in V}$ to $(\overline{r_u - r_v})_{e_{uv} \in E}$, with the orientation ensuring only one such difference per edge, and where $\overline{r_u - r_v}$ denotes the class of $r_u - r_v$ in $R/I_e$. The kernel of the potential morphism is a ring of generalized splines $R_G$. The image $\overline{R_G} := \im(\rho_G)$ is the ring of \textbf{generalized cosplines}, and write $\overline{\rho_G} : R^V \rightarrow \overline{R_G}$ for the morphism induced by $\rho_G$. We also have the \textbf{copotential morphism} $\underline{\rho_G} : R^E \rightarrow R^E / \overline{R_G}$ induced by ring quotient in the usual way. We define $\underline{R_G} := R^E/\overline{R_G}$ to be the ring of \textbf{failed cosplines}. Thus, we can write the copotential morphism as $\underline{\rho_G} : R^E \rightarrow \underline{R_G}$.
\end{definition}

In particular, when $R$ is Noetherian, the finiteness of $G$ implies that the modules $R^V$, $R^E$, $R_G$, $\overline{R_G}$, and $\underline{R_G}$ are finitely generated $R$-modules. These give rise to two fundamental exact sequences: the \textbf{potential exact sequence} and \textbf{copotential exact sequence}. Note that the potential exact sequence was first defined by DiPasquale in \textit{Generalized splines and graphic arrangements} \cite[Lemma 2.2]{dipasquale2017generalized}.

\begin{lemma}[The Fundamental Sequences] \label{lem:exact_sequences}
Let $R$ be a commutative ring and $G$ a finite edge-labeled graph. There exist canonical exact sequences:
\[
\begin{tikzcd}
0 \arrow[r] & R_G \arrow[r] & R^V \arrow[r, "\overline{\rho_G}"] & \overline{R_G} \arrow[r] & 0
\end{tikzcd}
\]
and
\[
\begin{tikzcd}
0 \arrow[r] & \overline{R_G} \arrow[r] & R^E \arrow[r, "\underline{\rho_G}"] & \underline{R_G} \arrow[r] & 0
\end{tikzcd}
\]
\end{lemma}
\begin{proof}
The given exact sequences are induced by the first isomorphism theorem for $R$-modules applied to the $R$-linear maps $\overline{\rho_G}$ and $\underline{\rho_G}$.
\end{proof}

We now consider the following propositions. They follow by applying the general homological local ring results to the potential and copotential exact sequences established in Lemma~\ref{lem:exact_sequences}.

\begin{proposition}[Depth Estimates] \label{prop:depth_estimates}
Let $R$ be a commutative Noetherian local ring with finite edge-labeled graph $G$. At the level of $R$-modules the following hold.
\begin{enumerate}
    \item \textbf{Potential Sequence Estimates:}
    \begin{itemize}
    \item $\depth(R_G) \geq \min\{\depth(R^V), \depth(\overline{R_G}) + 1 \}$
    \item $\depth(R^V) \geq \min\{\depth(R_G), \depth(\overline{R_G})\}$
    \item $\depth(\overline{R_G}) \geq \min\{\depth(R^V), \depth(R_G)-1\}$
    \item $\Ass(R_G) \subseteq \Ass(R^V)$
    \item $\Ass(R^V) \subseteq \Ass(R_G) \cup \Ass(\overline{R_G})$
    \item $\Ass(R_G \oplus \overline{R_G}) = \Ass(R_G) \cup \Ass(\overline{R_G})$
    \end{itemize}
    \item \textbf{Copotential Sequence Estimates:}
    \begin{itemize}
    \item $\depth(\overline{R_G}) \geq \min\{\depth(R^E), \depth(\underline{R_G}) + 1 \}$
    \item $\depth(R^E) \geq \min\{\depth(\underline{R_G}), \depth(\overline{R_G})\}$
    \item $\depth(\underline{R_G}) \geq \min\{\depth(R^E), \depth(\overline{R_G})-1\}$
    \item $\Ass(\overline{R_G}) \subseteq \Ass(R^E)$
    \item $\Ass(R^E) \subseteq \Ass(\overline{R_G}) \cup \Ass(\underline{R_G})$
    \item $\Ass(\overline{R_G} \oplus \underline{R_G}) = \Ass(\overline{R_G}) \cup \Ass(\underline{R_G})$
    \end{itemize}
    \item \textbf{Base Ring Properties:}
    \begin{itemize}
    \item $\depth(R^V) = \depth(R)$
    \item $\depth(R^E) = \min\{ \depth(R/I_e) \}_{e \in E}$ if $E \neq \emptyset$.
    \end{itemize}
\end{enumerate}
\end{proposition}
\begin{proof}
These follow immediately by the potential and copotential exact sequences combined with the standard results on depth and associated primes from the previous section. The finiteness of $V$ and $E$ ensures the finite sum formulas for depth apply.
\end{proof}

The next proposition applies only to Cohen--Macaulay Noetherian local rings, and is therefore somewhat more restrictive than the previous proposition.

\begin{proposition} \label{prop:cm_properties}
Let $R$ be a Cohen--Macaulay commutative Noetherian local ring with a finite edge-labeled graph $G$ where all edge labels are non-zero principal ideals generated by regular elements. Then:
\begin{itemize}
\item $\depth(R^V) = \dim(R)$
\item $\depth(R^E) = \dim(R) - 1$ (assuming $E \neq \emptyset$)
\end{itemize}
In particular, if $\dim(R) \geq 2$ and $E \neq \emptyset$, then $\mathfrak m \not \in \Ass(R^E)$, $\mathfrak m \not \in \Ass(\overline{R_G})$, and $\depth(\overline{R_G}) \geq 1$.
\end{proposition}
\begin{proof}
Since $R$ is Cohen--Macaulay $\depth(R) = \dim(R)$, so by Proposition~\ref{prop:depth_estimates} $\depth(R^V) = \dim(R)$. Furthermore, for any principal ideal $I_e$ with regular generator we have $\depth(R/I_e) = \depth(R) - 1$, so by Proposition~\ref{prop:depth_estimates} $\depth(R^E) = \depth(R)-1$.

Now suppose $\dim(R) \geq 2$ and $E \neq \emptyset$. Then $\depth(R^E) \geq 1$, so $\mathfrak m \not \in \Ass(R^E)$. Now since $\Ass(\overline{R_G}) \subseteq \Ass(R^E)$ (by Proposition~\ref{prop:depth_estimates}), we have $\mathfrak m \not \in \Ass(\overline{R_G})$. But that implies $\depth(\overline{R_G}) \geq 1$.
\end{proof}

Proposition~\ref{prop:cm_properties} illuminates a unique depth relationship between the potential cospline module and the cospline module. Normally, the depth of any one module in an exact sequence is regulated by the minimum of the depths of the other two modules. For low depth, however, the fact that associated primes are inherited by submodules provides a mechanism for the potential cospline module to directly regulate the depth of the cospline module, circumventing the failed cospline module entirely. Because the potential cospline module is determined entirely by the ideals themselves, and not the graph structure itself, the proposition highlights a particular situation where the graph structure is irrelevant. The following theorem continues the theme.

\begin{theorem} \label{thm:freeness}
Let $R$ be a Cohen--Macaulay Noetherian local ring of Krull dimension $2$. Let $G$ be a finite edge-labeled graph with non-zero principal edge labels generated by regular elements, such that the non-zero spline module $R_G$ has finite projective dimension. Then $R_G$ is free.
\end{theorem}
\begin{proof}
Let $R$ be a Cohen--Macaulay Noetherian local ring of Krull dimension $2$. If $E = \emptyset$, then $R_G \cong R^V \cong R^{|V|}$, which is free. Thus we assume $E \neq \emptyset$. Notice that $\depth(R_G) \geq \min\{ \depth(R^V), \depth(\overline{R_G}) +1 \}$ (by Proposition~\ref{prop:depth_estimates}). Now by Proposition~\ref{prop:cm_properties} and the fact $R$ is a Cohen--Macaulay Noetherian local ring, we have
\begin{align*}
\depth(R_G) &\geq \min\{ \depth(R^V), \depth(\overline{R_G}) +1 \}\\
\depth(R_G) &\geq \min\{ 2, \depth(\overline{R_G}) +1 \}.
\end{align*}

Now by Proposition~\ref{prop:cm_properties} and the hypothesis on principal edge labels, we know $\depth(\overline{R_G}) \geq 1$. Thus:
\[
\depth(R_G) \geq 2
\]

Now consider the Auslander--Buchsbaum formula, noting that since $R$ is Cohen--Macaulay we have $\depth(R) = \dim(R) = 2$ and by hypothesis $\pd_R(R_G) < \infty$:
\begin{align*}
\pd_R(R_G) + \depth(R_G) &= \depth(R) \\
\pd_R(R_G) + 2 = 2 \\
\pd_R(R_G) = 0
\end{align*}

So $R_G$ has projective dimension $0$ and must therefore be projective. But every finitely generated projective module over a local ring is free, so $R_G$ is a free $R$-module as claimed.
\end{proof}

The theorem identifies a precise condition where the only barrier to freeness is the structure of the ideals and the ring, not the graph itself.

\section{Applications}
The previous section was focused entirely on the local case. In this section, we obtain a global corollary by means of local--global arguments.

\begin{figure}[ht]
    \centering
    \begin{tikzpicture}[auto, thick]
        \node[circle, draw, fill=black!10, inner sep=2pt] (v1) at (0,0) {$v_1$};
        \node[circle, draw, fill=black!10, inner sep=2pt] (v2) at (3,0) {$v_2$};
        \node[circle, draw, fill=black!10, inner sep=2pt] (v3) at (3,3) {$v_3$};
        \node[circle, draw, fill=black!10, inner sep=2pt] (v4) at (0,3) {$v_4$};

        \draw (v1) -- node[below] {$\langle y \rangle$} (v2);
        \draw (v2) -- node[right] {$\langle x \rangle$} (v3);
        \draw (v3) -- node[above] {$\langle y \rangle$} (v4);
        \draw (v4) -- node[left] {$\langle x \rangle$} (v1);
        
        \draw (v1) -- node[fill=white, inner sep=1pt] {$\langle x+y \rangle$} (v3);
    \end{tikzpicture}
    \caption{A diamond graph with principal ideal edge labels over $k[x,y]$. Theorem~\ref{thm:main_theorem} guarantees the spline module is free.}
    \label{fig:diamond}
\end{figure}
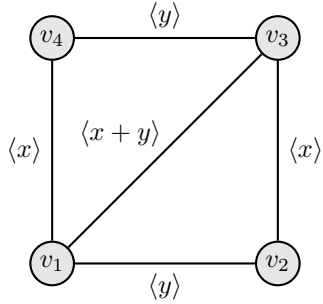

\begin{theorem}\label{thm:main_theorem}
Suppose $R$ is a Cohen--Macaulay Noetherian domain with Krull dimension $2$ at each maximal ideal of $R$. If $G$ is a finite edge-labeled graph with non-zero principal edge labels such that $R_G$ has finite projective dimension, then $R_G$ is a finite projective $R$-module. In particular, if finite projective $R$-modules are also free, then $R_G$ is free.
\end{theorem}
\begin{proof}
Since $R$ is a domain, every non-zero principal ideal is generated by a regular element. Let $\mathfrak m$ be a maximal ideal of $R$. Now by the stability of spline structure morphisms under flat base change \cite[Main Theorem]{stoltz2025scheme} we have that $(R_G)_{\mathfrak m} \cong (R_{\mathfrak m})_{G_{\mathfrak m}}$. That is, $(R_G)_{\mathfrak m}$ is isomorphic to a module of splines over the local ring $R_{\mathfrak m}$ that has edge-labeled graph $G_{\mathfrak m}$. By hypothesis $(R_G)_{\mathfrak m} \cong(R_{\mathfrak m})_{G_{\mathfrak m}}$ satisfies the conditions of Theorem~\ref{thm:freeness}, and is therefore free.

Furthermore, each $(R_G)_{\mathfrak m} \cong (R_{\mathfrak m})_{G_{\mathfrak m}}$ is the kernel of a potential short exact sequence, for which the middle term is free, and since each $R_{\mathfrak m}$ is Noetherian \cite[Lemma 10.31.1]{stacks-project} we have that $(R_G)_{\mathfrak m}$ is finitely generated. Now for modules over Noetherian rings, finitely generated implies finitely presented \cite[Lemma 10.31.4]{stacks-project} and therefore each $(R_G)_{\mathfrak m} \cong (R_{\mathfrak m})_{G_{\mathfrak m}}$ is also finitely presented. 

So we have established that for each maximal ideal $\mathfrak m$, we have that $(R_G)_{\mathfrak m}$ is free and of finite presentation. Now any $R$-module that is free and of finite presentation locally at each maximal ideal of $R$ is also finite projective \cite[Lemma 10.78.2(5)]{stacks-project}, so it follows that $R_G$ is a finite projective $R$-module as claimed.
\end{proof}

As a corollary, we obtain that the following rings have a free module of splines for any finite edge-labeled graph with principal ideal labels.

\begin{corollary}
Let $R$ be either $k[x,y]$ or $k[x_1,...x_n]/\langle f \rangle$, with $f=0$ defining an irreducible nonsingular dimension $2$ surface, and $k$ a field. For any edge-labeled graph $G$ over $R$ with principal edge labels, the module of splines $R_G$ is free.
\end{corollary}
\begin{proof}
Let  $R$, $G$, and $R_G$ be given as specified. Note that by hypothesis $R$ is the coordinate ring of a nonsingular surface, and is therefore a regular ring \cite[Definition 28.9.1, Lemma 28.9.2(3)]{stacks-project}. Since regular rings are Cohen--Macaulay \cite[Example 2.2]{mathew_homological}, we can see $R$ is Cohen--Macaulay. We also know that $R$ is Krull dimension $2$ at each maximal ideal by hypothesis. Furthermore, since $R$ is a regular ring, it has finite global dimension, which implies that $R_G$ has finite projective dimension \cite[Lemma 10.110.8, Definition 10.109.2]{stacks-project}. Finally, by the irreducibility hypothesis we know $R$ is an integral domain, and it is Noetherian since the quotient of a Noetherian ring is Noetherian. So the criteria of Theorem~\ref{thm:main_theorem} are satisfied and the claim follows.
\end{proof}

Consequently, every module of multivariate generalized splines in two variables is free, provided the edge-labeled graph has principal edge labels; Alt{\i}nok and Sar{\i}o\u{g}lan \cite{altinok2024multivariate} previously showed every cycle with principal edge labels had a free spline module, and the author later proved that freeness held when the graph was locally trivial or determined by a cycle \cite{stoltz2025scheme}. The corollary now settles the matter for $k[x,y]$ entirely, as by the Quillen--Suslin theorem projective modules over polynomial rings are free. Note that the author was motivated to use projective dimension locally after seeing the success of Alt{\i}nok and Sar{\i}o\u{g}lan's global homological approach in the cycle case \cite{altinok2024multivariate}. Thus, the primary application is, in spirit, a local analog of Alt{\i}nok and Sar{\i}o\u{g}lan's work, and effectively generalizes the original Billera and Rose local--global approach on stars \cite[Theorem 2.3]{billera1991modules}.

\section{Conclusion}
We have established the basics of a theory of generalized splines over Noetherian local rings from the homological perspective. In particular, that the potential and copotential exact sequences are a homological control mechanism afforded to the generalized spline structure. These exact sequences allow manipulation of depth, which features as an input into the Auslander--Buchsbaum formula and thereby regulates projectivity of generalized spline modules locally.

In the case of low depth rings, the inheritance of associated primes by submodules accounts for the indifference of certain generalized spline modules to the graph structure. That is because when the potential cospline module has depth $1$, the depth of the cospline submodules (picked out by a particular configuration of edges) is constrained above $0$ by the absence of an associated prime in the parent module. Thus, this is often enough to force the Auslander--Buchsbaum formula to induce a projective dimension of $0$ in the local situation, and thereby force local freeness. When the depth of a ring swells, projective dimension $0$ requires an ever greater depth for the spline module, which provides more opportunities for the graph structure to thwart local projectivity.

Currently, we have not considered how graph features regulate the depth of the cospline and failed cospline modules. For example, it is known that certain graphs, such as cycles with certain well-behaved degree, have free spline modules \cite{altinok2024multivariate}. It is not clear how this manifests itself with respect to the depth of the modules that appear within the potential and copotential exact sequences. Other avenues for study include the use of well-established cohomological tools, such as spectral sequences, which we did not consider.


\end{document}